\include{kbordermatrix.sty}
\documentclass[12pt]{amsart}
\usepackage{amsmath,amssymb,enumerate}
\newtheorem{theorem}{Theorem}[section]
\newtheorem{claim}{}[theorem]
\newtheorem{lemma}[theorem]{Lemma}

\newtheorem{corollary}[theorem]{Corollary}
\newtheorem{conjecture}[theorem]{Conjecture}
\theoremstyle{definition}

\newcommand{\bR}{\mathbb R}
\newcommand{\bZ}{\mathbb Z}

\newcommand{\cF}{\mathcal{F}}

\newcommand{\cL}{\mathcal{L}}

\newcommand{\cU}{\mathcal{U}}

\DeclareMathOperator{\si}{si}

\DeclareMathOperator{\cl}{cl}

\DeclareMathOperator{\PG}{PG}
\DeclareMathOperator{\GF}{GF}
\DeclareMathOperator{\AG}{AG}

\newcommand{\del}{\setminus \!}
\newcommand{\con}{/}
\newcommand{\gbinom}[3]{{{#1} \brack {#2}}_{#3}}
\newcommand{\qbinom}[2]{\gbinom{#1}{#2}{q}}
\newcommand{\tqbinom}[2]{\textstyle{\qbinom{#1}{#2}}}
\allowdisplaybreaks[1]

\author{Peter Nelson}
\address{Department of Combinatorics and Optimization,
University of Waterloo, Waterloo, Canada}
\thanks{ This research was partially supported by a grant from the
Office of Naval Research [N00014-12-1-0031].}
\title[The number of flats in a matroid]{The Number of rank-$k$ Flats in a Matroid with no $U_{2,n}$-minor}
\begin{document}
\begin{abstract}
We show that, if $k$ and $\ell$ are positive integers and $r$ is sufficiently large, then the number of rank-$k$ flats in a rank-$r$ matroid $M$ with no $U_{2,\ell+2}$-minor is less than or equal to number of rank-$k$ flats in a rank-$r$ projective geometry over $\GF(q)$, where $q$ is the largest prime power not exceeding $\ell$. 
\end{abstract}
\maketitle

\section{Introduction}

Let $W_k(M)$ denote the number of rank-$k$ flats in a matroid $M$. For example, we have $W_k(\PG(r-1,q)) = \qbinom{r}{k}$, the $q$-binomial coefficient for $r$ and $k$. The following conjecture appears in Oxley [\ref{oxley} p. 582], attributed to Bonin:

\begin{conjecture}\label{mainconj}
	If $q$ is a prime power, $k \ge 0$ is an integer and $M$ is a rank-$r$ matroid with no $U_{2,q+2}$-minor, then $W_k(M) \le \qbinom{r}{k}$. 
\end{conjecture}

Unfortunately for $k = 2$, $r = 3$ this conjecture is false for all $q \ge 13$; we discuss counterexamples due to Blokhuis (private communication) soon. Our main theorem, on the other hand, resolves the conjecture whenever $r$ is large compared to $q$ and $k$. In fact we show more, obtaining an eventually best-possible bound on $W_k$ when excluding an arbitrary rank-$2$ uniform minor: 

\begin{theorem}\label{main}
	Let $\ell \ge 2$ and $k \ge 0$ be integers. If $r$ is sufficiently large and $M$ is a rank-$r$ matroid with no $U_{2,\ell+2}$-minor, then $W_k(M) \le \qbinom{r}{k}$, where $q$ is the largest prime power so that $q \le \ell$.\end{theorem}

This was shown for $k = 1$ in [\ref{gn}]. The bound is attained by projective geometries over $\GF(q)$, so cannot be improved.  

Our theorem does not resolve Conjecture~\ref{mainconj} in the case where $r$ is not too large compared to $k$; in particular, the conjecture remains open in the interesting case when $k = r-1$ (that is, where $W_k$ is the number of hyperplanes).

 We now discuss the counterexamples for $r = 3$ and $k = 2$, first giving the construction of Blokhuis. For each simple rank-$3$ matroid $M$, let $\cL^+(M)$ be the set of lines of $M$ containing at least $3$ points. Note that $\cL^+(M)$ determines $M$.

\begin{lemma}\label{counterexample}
	If $q$ is a prime power then there is a rank-$3$ matroid $M(q)$ with no $U_{2,2q}$-minor such that $W_2(M(q)) = \tfrac{1}{2}q^2(q+1)$. 
\end{lemma}
\begin{proof}
 Let $N \cong \AG(2,q)$. Let $\cL$ be a set of $q$ pairwise disjoint lines of $N$. If $M(q)$ is the simple rank-$3$ matroid with $E(M(q)) = E(N)$ and $\cL^+(M(q)) = \cL^+(N) - \cL$, then $M(q)$ has $q^2 + q\binom{q}{2} = \tfrac{1}{2}q^2(q+1)$ lines and each element of $M(q)$ lies on $2q-1$ lines of $M(q)$, so $M(q)$ has no $U_{2,2q}$-minor.   
\end{proof}

We now verify that, when $r = 3$ and $k = 2$, Conjecture~\ref{mainconj} is false for nearly all $q$:
\begin{corollary}
	Let $q > 125$ be a prime power. There is a rank-$3$ matroid $M$ with no $U_{2,q+2}$-minor such that $W_2(M) > \qbinom{3}{2}$.
\end{corollary}
\begin{proof}
	Let $q'$ be a power of $2$ so that $\tfrac{1}{4}(q+2) < q' \le \tfrac{1}{2}(q+2)$. Now $M(q')$ of Lemma~\ref{counterexample} has no $U_{2,2q'}$-minor so has no $U_{2,q+2}$-minor, and \[W_2(M(q')) = \tfrac{1}{2}(q')^2(q' + 1) > \tfrac{1}{128}(q+2)^3 \ge (q+2)^2 > q^2 + q + 1 = \tqbinom{3}{2}.\]
\end{proof}

	If more care is taken, then the same construction can in fact be shown to provide counterexamples for all $q \ge 13$. Smaller values of $q$ will be considered in detail in a future paper. 
	
	Despite these examples, it is likely that the rank-$3$ case is sporadic and that Conjecture~\ref{mainconj} holds unconditionally for all $r \ge 4$. We also conjecture a strengthened version of Theorem~\ref{main}, in which $r$ is not required to be large compared to $k$:

\begin{conjecture}
	Let $\ell \ge 2$ be an integer. If $r$ is sufficiently large and $M$ is a rank-$r$ matroid with no $U_{2,\ell+2}$-minor, then $W_k(M) \le \qbinom{r}{k}$ for all integers $k \ge 0$, where $q$ is the largest prime power such that $q \le \ell$. 
\end{conjecture}	

\section{Preliminaries}

We follow the notation of Oxley [\ref{oxley}]. In particular for each integer $\ell \ge 0$ we write $\cU(\ell)$ for the class of matroids with no $U_{2,\ell+2}$-minor. 

The first theorem we need gives a bound on $W_1$ for all matroids in $\cU(\ell)$, and was proved by Kung [\ref{kungextremal}]. Note that this resolves the $k = 1$ case of Conjecture~\ref{mainconj}. 
\begin{theorem}\label{kung}
	If $\ell \ge 2$ and $M \in \cU(\ell)$, then $W_1(M) \le \frac{\ell^{r(M)}-1}{\ell-1}$. 
\end{theorem}
	We often use the cruder bound $W_1(M) < \ell^{r(M)}$. The next result, which provides a large affine geometry restriction in a dense matroid in $\cU(\ell)$ of very large rank, appears in [\ref{gn2}]. 

\begin{theorem}\label{hj}
	There is a function $f: \bZ^3 \times \bR \rightarrow \bZ$ such that, for all $n,\ell \in \bZ^+$, $\alpha \in \bR^+$ and prime powers $q$, if $M \in \cU(\ell)$ satisfies $W_1(M) \ge \alpha q^{r(M)}$ and $r(M) \ge f(\ell,n,q,\alpha)$, then $M$ has either an $\AG(n,q)$-restriction or a $\PG(n,q')$-minor for some $q' > q$. 
\end{theorem}

We now consider the parameter $W_k(M)$, known as the \emph{$k$-th Whitney number of $M$ of the second kind}, and its value on projective geometries. It is well-known (see [\ref{oxley} p.162], for example) that $\PG(r-1,q)$ has exactly $\qbinom{r}{k}$ rank-$k$ flats, where $\qbinom{r}{k}$ is the `$q$-binomial coefficient' defined recursively by $\qbinom{r}{0} = \qbinom{r}{r} = 1$ and $\qbinom{r}{k} = q^k\qbinom{r-1}{k} + \qbinom{r-1}{k-1}$ for $0 < k < r$. An equivalent definition is given by 
\[\qbinom{r}{k} = \frac{(q^r-1)(q^{r-1}-1)\dotsc(q^{r-k+1}-1)}{(q^k-1)(q^{k-1}-1)\dotsc(q-1)}.\] 
	Using these definitions, it is not hard to show that $\qbinom{r}{k}$ satisfies a few basic properties, which we will use freely:

\begin{lemma}\label{qbinomproperties}
	For every prime power $q$ and all integers $0 < k < r$, the following hold:
	\begin{enumerate}
		\item\label{qb1} $\qbinom{r}{k} \ge q^{ki}\qbinom{r-i}{k}$ for all $i \in \{0, \dotsc, r\}$. 
		\item\label{qb2} $q^{k(r-k)} \le \qbinom{r}{k} \le q^{rk}$.
		\item\label{qb3} $\qbinom{r}{k} = \qbinom{r-1}{k} + q^{r-k}\qbinom{r-1}{k-1}.$
	\end{enumerate}
\end{lemma}
 We now consider $W_k(M)$ for a general matroid $M$. For each $e \in E(M)$ let $\cF_k(M;e)$ denote the set of rank-$k$ flats of $M$ containing $e$, and let $W_k^e(M) = W_k(M) - |\cF_k(M;e)|$ denote the number of rank-$k$ flats of $M$ \emph{not} containing $e$. We will also freely use some basic properties of $W_k$:
 	\begin{lemma}\label{sigmaproperties}
		If $k \ge 1$ and $\ell \ge 2$ are integers, $M$ is a matroid, and $e$ is a nonloop of $M$ then the following hold: 
		\begin{enumerate}
			\item\label{sp1} $W_k(M) \le W_1(M)^k$. 
			\item\label{sp2} $W_k(M) < \ell^{kr(M)}$ if $M \in \cU(\ell)$.
			\item\label{sp3} $|\cF_k(M;e)| = W_{k-1}(M \con e)$.
			\item\label{sp4} $W_k(M) = W_{k-1}(M \con e) + \sum_{F \in \cF_{k+1}(M;e)}W_k^e(M|F)$. 
		\end{enumerate}
	\end{lemma}

\begin{proof}
	(\ref{sp1}) follows from the fact that every rank-$k$ flat is spanned by $k$ points, and (\ref{sp2}) follows from (1) and Theorem~\ref{kung}. (\ref{sp3}) is easy. Now by (\ref{sp3}), there are $W_{k-1}(M \con e)$ rank-$k$ flats of $M$ containing $e$. For each other rank-$k$ flat $F'$ of $M$, the set $F = \cl_M(F' \cup \{e\})$ is the unique rank-$(k+1)$ flat of $M$ containing $e$ and $F'$, and each such $F$ corresponds to $W_k^e(M|F)$ different $F'$. Combining these statements gives (\ref{sp4}).
\end{proof}

\section{Geometry}

In this section, we deal with projective and affine geometries over $\GF(q)$, using them to provide a $U_{2,q^2+1}$-minor in various situations. We repeatedly use the fact that, if $M$ has an $\AG(r(M)-1,q)$-restriction $R$ and $e \in E(R)$, then $M \con e$ has a $\PG(r(M \con e)-1,q)$-restriction contained in $E(R)$. The first lemma we need was also essentially proved in [\ref{gn}].

\begin{lemma}\label{spanningpg}
	If $q$ is a prime power and $M$ is a simple matroid of rank at least $3$ with a proper $\PG(r(M)-1,q)$-restriction, then $M$ has a $U_{2,q^2+1}$-minor. 
\end{lemma}
\begin{proof}
	Let $R$ be a $\PG(r(M)-1,q)$-restriction of $M$. We may assume that $E(M) = E(R) \cup \{e\}$ for some $e \notin E(R)$. The point $e$ is spanned by at most one line of $R$; by repeatedly contracting points not on such a line and simplifying we obtain a simple rank-$3$ minor of $M'$ such that $E(M') = E(R') \cup \{e\}$ and $R' \cong \PG(2,q)$. Now $e$ is spanned by at most one line of $R'$ and such a line contains $q+1$ elements of $E(R')$, so $W_1(M' \con e) \ge |E(R')| - q = q^2 + 1$, and so $M' \con e$ has a  $U_{2,q^2+1}$-restriction. 
	\end{proof}

	In particular, if $M$ has rank at least $3$, has a $\PG(r(M)-1,q)$-restriction and is not $\GF(q)$-representable then $M$ has a $U_{2,q^2+1}$-minor; we use this idea in the next two lemmas. 
	
	\begin{lemma}\label{win}
	Let $q$ be a prime power and $m \ge 2$ and $b \ge 1$ be integers. If $M$ is a matroid with an $\AG(m+b,q)$-restriction $R$, a rank-$m$ restriction $S$ that is not $\GF(q)$-representable, and every cocircuit of $M$ has rank at least $r(M)-b$, then $M$ has a $U_{2,q^2+1}$-minor. 
	\end{lemma}
\begin{proof} 
	We may assume that no minor of $M$ satisfies the hypotheses. Note that contracting elements of $M$ preserves the cocircuit property, 
so $E(M) = \cl_M(E(R)) \cup \cl_M(E(S))$. If $r(M) > r(R)$ then $E(M) - \cl_M(E(R))$ contains a cocircuit of $M$ of rank at most $r(S) = m < r(M) - b$, a contradiction. Therefore $R$ is spanning in $M$. 
Let $f \in E(R) - \cl_M(E(S))$; the matroid $M \con f$ has a $\PG(r(M \con f)-1,q)$-restriction, has rank at least $3$ and is not $\GF(q)$-representable, so has a $U_{2,q^2+1}$-minor by Lemma~\ref{spanningpg}. 
\end{proof}


\begin{lemma}\label{spanningwin}
	Let $q$ be a prime power and $k \ge 1$ be an integer. If $M$ is a matroid such that $r(M) \ge k+3$, $M$ has an $\AG(r(M)-1,q)$-restriction and $M$ has no $U_{2,q^2+1}$-minor, then  $W_k(M) \le \qbinom{r(M)}{k}$. 
\end{lemma}
\begin{proof}
	Let $R$ be an $\AG(r(M)-1,q)$-restriction of $M$. We may assume that $M$ is simple. We make two claims, considering two different types of rank-$k$ flat. 
	
	\begin{claim}\label{fc}
		If $F$ is a flat of $M$ with $F \cap E(R) \ne \varnothing$, then $F$ has a basis contained in $E(R)$. 
	\end{claim}
	\begin{proof}[Proof of claim:]
	For each $e \in E(R)$, the matroid $M \con e$ has rank at least $3$ and has a $\PG(r(M)-2,q)$-restriction contained in $E(R)-\{e\}$, so it follows from Lemma~\ref{spanningpg} that, for every $e \in E(R)$, each nonloop of $M \con e$ is parallel in $M \con e$ to some element of $E(R)-\{e\}$. Therefore every $x \in E(M)$ is in some line of $M$ containing $e$ and another element $y$ of $E(R)$. Thus, if $F$ is a flat of $M$ and $e \in F \cap E(R)$, then $F$ has a basis contained in $E(R)$, as we can include $e$, and then can exchange each $x \in F - E(R)$ with its corresponding $y \in E(R)$. 
	\end{proof}
	
	\begin{claim}\label{sc}
		If $F$ is a rank-$k$ flat of $M$ such that $F \cap E(R) = \varnothing$, then $F$ is a rank-$k$ flat of $M \con e\del (E(R)-\{e\})$ for all $e \in E(R)$. 
	\end{claim}
	\begin{proof}[Proof of claim:]
	Let $F$ be a rank-$k$ flat of $M$ that is disjoint from $E(R)$ and let $e \in E(R)$. Let $F' = \cl_M(F \cup \{e\})$. By the first claim, $F'$ contains a rank-$(k+1)$ flat $G$ of $R$; note that $R|G \cong \AG(k,q)$. If $F' = F \cup G$ then the claim holds. Otherwise, $F' \ne F \cup G$ and $F'$ is the disjoint union of a rank-$(k+1)$ affine geometry, a rank-$k$ flat, and at least one other point, so $M|F'$ is not $\GF(q)$-representable. Let $f \in E(R) - F'$. The matroid $M \con f$ has rank at least $3$, has a $\PG(r(M\con f)-1,q)$-restriction contained in $E(R)$ and has $M|F'$ as a restriction, so Lemma~\ref{spanningpg} gives a contradiction. 
	\end{proof}
	Let $e \in E(R)$. By~\ref{fc}, the number of rank-$k$ flats of $M$ that intersect $E(R)$ is $W_k(R)$. By~\ref{sc}, the number of other rank-$k$ flats of $M$ is at most $W_k(M \con e \del E(R))$. Now $M \con e$ has rank at least $3$ and has a $\PG(r(M)-2,q)$-restriction, so we may assume by Lemma~\ref{spanningpg} that $\si(M \con e) \cong \PG(r(M)-2,q)$ and so $M \con e \del E(R)$ is $\GF(q)$-representable. Therefore \begin{align*}
	W_k(M) &\le W_k(R) + W_k(M \con e \del E(R)) \\
	& \le W_k(\AG(r(M)-1,q)) + W_k(\PG(r(M)-2),q).
\end{align*}
	This upper bound is clearly equal to $W_k(\PG(r(M)-1,q)) = \qbinom{r}{k}$. 
	
	\end{proof}

\section{The Main Theorem}

We now restate and prove Theorem~\ref{main}. 

\begin{theorem}
	There is a function $g: \bZ^2 \rightarrow \bZ$ so that, for all integers $\ell \ge 2$ and $k \ge 0$, if $M \in \cU(\ell)$ satisfies $r(M) \ge g(\ell,k)$ then $W_k(M) \le \qbinom{r(M)}{k}$, where $q$ is the largest prime power not exceeding $\ell$. 
\end{theorem}
\begin{proof}
	Set $g(\ell,0) = 0$ for all $\ell$; note that this trivially satisfies the conditions of the theorem. Let $\ell \ge 2$ and $k > 0$ be integers, and $q$ be the largest prime power such that $q \le \ell$. If $\ell = 2$ then $M$ is binary and the bound is obvious; we may therefore assume that $\ell \ge q \ge 3$. Suppose recursively that $g(\ell,i)$ has been defined for each $i \in \{0, \dotsc, k-1\}$.  Let $r_0  = \max(k+3,\max_{0 \le i \le k - 1}g(\ell,i))$. Note that $2q^{-k} \le \tfrac{2}{3}$; let $b$ be a positive integer so that $kq^{k^2-b} + (2 q^{-k})^{b+1} \le \tfrac{1}{6}\ell^{-k(k+1)}$.  
	Recall that the function $f$ was defined in Theorem~\ref{hj}; set $g(\ell,k)$ to be an integer $n$ such that $q^{-k^2}2^n > \ell^{k f(\ell,r_0+b,q,q^{-k})}$. 
	
	Suppose inductively that $g(\ell,k-1)$ satisfies the theorem statement. If $g(\ell,k)$ does not, then there exists $M_0 \in \cU(\ell)$ such that $r(M_0) \ge n$ and $W_k(M_0) > \qbinom{r(M_0)}{k}$. We will obtain a contradiction by finding a $U_{2,\ell+2}$-minor of $M$; since $q^2 + 1 \ge \ell + 2$ it is also enough to find a $U_{2,q^2+1}$-minor.  
	
	 Let $M$ be minor-minimal such that $M$ is a minor of $M_0$ and $W_k(M) > 2^{r(M_0) - r(M)}\qbinom{r(M)}{k}$. Note that $M$ is simple; let $r = r(M)$. We often use the fact that $W_k(M') < (2 q^{-k})^{r-r(M')}W_k(M)$ for each proper minor $M'$ of $M$, which follows from minimality and (\ref{qb1}) of Lemma~\ref{qbinomproperties}. 
	\begin{claim}\label{decency}
	 $M$ has an $\AG(r_0+b,q)$-restriction.
	\end{claim}
	\begin{proof}[Proof of claim:]
		
		Observe that \[W_k(M) > 2^{r(M_0) - r}\tqbinom{r}{k} \ge 2^{n-r}q^{k(r-k)} = q^{-k^2}2^{n}(q^k/2)^r > \ell^{kf(\ell,r_0+b,q,q^{-k})},\] so $r > f(\ell,r_0+b,q,q^{-k}).$ By choice of $M$ and Lemmas~\ref{qbinomproperties} and~\ref{sigmaproperties} we have $W_1(M)^k \ge W_k(M) > \qbinom{r}{k} \ge q^{k(r-k)}$, so $W_1(M) \ge q^{-k}q^r$. The required restriction exists by Theorem~\ref{hj}, since $\PG(r_0+b,q')$ has a $U_{2,\ell+2}$-minor for all $q' > q$. 
	\end{proof}
			
		\begin{claim}\label{strongconn}
			Every cocircuit of $M$ has rank at least $r-b$. 
		\end{claim}
		\begin{proof}[Proof of claim:] Suppose not; let $C$ be a cocircuit of $M$ of rank less than $r-b$, let $H$ be the hyperplane $E(M)-C$, and let $B$ be a rank-$(r-b)$ set containing $C$. Note that $E(M) = H \cup B$.
					
			Let $e \in C$; note that the matroid $M\con e$ has no loops and that $r((M \con e)|(B-e)) = r - (b+1) \ge r_0$.  Let $\cF_B$ be the collection of rank-$k$ flats of $M \con e$ that intersect $B$. Each $F \in \cF_B$ is the closure of the union of a rank-$i$ flat of $(M \con e)|(B-\{e\})$ and a rank-$(k-i)$ flat of $(M\con e)|H$ for some $i \in \{1, \dotsc, k\}$, so 
			\begin{align*}
				|\cF_B| &\le \sum_{i = 1}^{k-1} W_i((M \con e)|(B-e))W_{k-i}((M \con e)|H) + W_k((M \con e)|(B-e))\\
				&\le \sum_{i = 1}^{k-1} \qbinom{r-b-1}{i} \qbinom{r-1}{k-i} + (2 q^{-k})^{b+1}W_k(M) \\
				& \le \sum_{i = 1}^{k-1} q^{i(r-b-1) + (k-i)(r-1)} + (2 q^{-k})^{b+1}W_k(M)\\
				& \le k q^{-b}q^{k(r-1)} + (2 q^{-k})^{b+1}W_k(M) \\
				& \le k q^{k^2-b}\qbinom{r}{k} + (2 q^{-k})^{b+1}W_k(M) \\
				&< \left(kq^{k^2-b} + (2 q^{-k})^{b+1}\right)W_k(M)\\
				&\le \tfrac{1}{6}\ell^{-k(k+1)}W_k(M).
			\end{align*}	
			For each rank-$k$ flat $F_0$ of $M \con e$ that is not in $\cF_B$, we have $F_0 \subseteq H$ so $(M \con e)|F_0 = M|F_0$. The closure in $M$ of $F = F_0 \cup \{e\}$ contains no elements of $B- \{e\}$, so $F \in \cF_{k+1}(M;e)$ and $W_k^e(M|F) = 1$. For each other $F \in \cF_{k+1}(M;e)$ we have $W_k^e(M|F) < \ell^{k(k+1)}$ by Lemma~\ref{sigmaproperties}. Therefore
			\begin{align*}
				\sum_{F \in \cF_{k+1}(M;e)} W_k^e(M|F)  &\le \ell^{k(k+1)}|\cF_B| + (W_k(M \con e) - |\cF_B|) \\ 
				& < \ell^{k(k+1)}|\cF_B| + 2 q^{-k} W_k(M) \\ 
				& \le \ell^{k(k+1)}\left(\tfrac{1}{6}\ell^{-k(k+1)}W_k(M)\right) + \tfrac{2}{3} W_k(M)\\
				& = \tfrac{5}{6}W_k(M).
			\end{align*}
			
		Now,  since $r(M\con e) \ge r_0$, by what is above we have
		\begin{align*}
			W_k(M) &= W_{k-1}(M \con e) + \sum_{F \in \cF_{k+1}(M;e)}W_k^e(M|F) \\
			& < \qbinom{r-1}{k-1} + \tfrac{5}{6}W_k(M)\\
			&< q^{k-r}\qbinom{r}{k} + \tfrac{5}{6}W_k(M),
		\end{align*}
		a contradiction, as $\tqbinom{r}{k} < W_k(M)$ and $q^{k-r} \le q^{k-r_0} \le q^{-3} < \tfrac{1}{6}$.
		\end{proof}	
			
		Let $N$ be a minor-minimal minor of $M$ such that 
		\begin{enumerate}
			\item\label{rest} $N$ has an $\AG(r_0+b,q)$-restriction, 
			\item\label{cocir} every cocircuit of $N$ has rank at least $r(N) - b$, and
			\item\label{dens} $W_k(N) > \qbinom{r(N)}{k}. $	
		\end{enumerate}	
		Let $R$ be an $\AG(r_0+b,q)$-restriction of $N$. Since $r_0 \ge k+1$, we may assume by \ref{decency}, \ref{strongconn} and Lemma~\ref{win} that every rank-$(k+1)$-restriction of $N$ is $\GF(q)$-representable. Note that $N$ has no loops.  
			
			\begin{claim}\label{contractlowerbound}	
			$W_k(N \con e) > \qbinom{r(N)-1}{k}$ for all $e \in E(N)$.	\end{claim}
		\begin{proof}[Proof of claim:]
		
		Since every rank-$(k+1)$ restriction of $N$ is $\GF(q)$-representable, the value of $W_k^e(N|F)$ for each rank-$(k+1)$ flat $F$ does not exceed $q^k$, its value on $\PG(k,q)$. Therefore
		$\sum_{F \in \cF_{k+1}(N;e)}W_k^e(N|F) \le q^k|\cF_{k+1}(N;e)| = q^kW_k(N \con e),$ and so by (4) of Lemma~\ref{sigmaproperties} we get 
		$W_k(N) \le W_{k-1}(N \con e) + q^kW_k(N \con e).$ Now $r(N \con e) \ge r_0$ so $W_{k-1}(N \con e) \le \qbinom{r(N)-1}{k-1}$ by the inductive hypothesis, and $W_k(N) > \qbinom{r(N)}{k}$, which implies that $W_k(N \con e) > q^{-k}\left(\qbinom{r(N)}{k} - \qbinom{r(N)-1}{k-1}\right) = \qbinom{r(N)-1}{k}.$ 
		\end{proof}
		
		Thus, properties (\ref{rest}) and (\ref{cocir}) and (\ref{dens}) are all preserved by contracting elements of $E(N) - \cl_N(E(R))$, so it follows from minimality that $R$ is spanning in $N$. We now obtain a contradiction from Lemma~\ref{spanningwin}. 
		\end{proof}
			
\section{Acknowledgements}

I thank Joe Bonin, Jim Geelen, Joseph Kung and the referees for their very useful feedback on the manuscript. 

\section{References}

\newcounter{refs}

\begin{list}{[\arabic{refs}]}
{\usecounter{refs}\setlength{\leftmargin}{10mm}\setlength{\itemsep}{0mm}}

\item\label{gn}
J. Geelen, P. Nelson, 
The number of points in a matroid with no $n$-point line as a minor, 
J. Combin. Theory. Ser. B 100 (2010), 625--630.

\item\label{gn2}
J. Geelen, P. Nelson, 
A density Hales-Jewett theorem for matroids, 
arXiv:1210.4522 [math.CO].

\item\label{kungextremal}
J.P.S. Kung,
Extremal matroid theory, in: Graph Structure Theory (Seattle WA, 1991), 
Contemporary Mathematics 147 (1993), American Mathematical Society, Providence RI, ~21--61.

\item \label{oxley}
J. G. Oxley, 
Matroid Theory, Second edition,
Oxford University Press, New York, 2011.

\end{list}
\end{document}